\documentclass[11pt,reqno]{amsart}
\usepackage{amsmath,amssymb}
\usepackage{mathrsfs}
\usepackage{amsfonts}
\usepackage[dvipsnames]{xcolor}
\usepackage{graphicx}
\usepackage{float}  
\usepackage{cite}
\usepackage{cases}
\usepackage{indentfirst}
\allowdisplaybreaks

\newtheorem{theorem}{Theorem}[section]
\newtheorem{lemma}[theorem]{Lemma}

\numberwithin{equation}{section}

\makeatletter
\usepackage[colorlinks=true,citecolor=blue,linkcolor=blue]{hyperref}

\begin{document}
\author{Yasheng Lyu}

\title{\textbf{Remarks on the Pogorelov type estimate for the degenerate $k$-Hessian equation} }

\address{School of Mathematics and Statistics, Xi'an Jiaotong University, Xi'an, Shaanxi 710049, People's Republic of China}

\email{lvysh21@stu.xjtu.edu.cn}

\begin{abstract}
This paper investigates the Pogorelov type estimate for the $k$-Hessian equation under a new condition on the degenerate right-hand side $f$. 
\end{abstract}


 \date{}
\maketitle

\pagestyle{myheadings}
\markboth{$~$ \hfill{\uppercase{Yasheng Lyu}}\hfill $~$}{$~$ \hfill{\uppercase{Pogorelov type estimate for the degenerate $k$-Hessian equation}}\hfill $~$}

\section{Introduction}

In this paper, we study the degenerate $k$-Hessian equation 
\begin{equation}\label{eqn1.1}
\sigma_{k}\big(\lambda\big(D^{2}u\big)\big)=f,\quad f\geq0,\quad \text{in}\ \Omega,
\end{equation}
for $5\leq k\leq n$.
Here, $\Omega$ is a bounded domain in $\mathbb{R}^{n}$, $\lambda(D^{2}u)$ denotes the eigenvalues of the Hessian matrix of $u$, and $\sigma_{k}(\lambda)$ is the $k$-th elementary symmetric polynomial on $\mathbb{R}^{n}$, defined by  
\[
\sigma_{k}(\lambda):=\sum_{1\leq i_{1}<\cdots<i_{k}\leq n}\lambda_{i_{1}}\cdots\lambda_{i_{k}},\quad \text{for}\ k=1,2,\dots,n. 
\]

Following the work of Caffarelli-Nirenberg-Spruck \cite{CNS1985}, a function $u\in C^{2}(\Omega)\cap C^{0}(\overline{\Omega})$ is called $k$-admissible if its Hessian eigenvalues satisfy  
\[
\lambda\big(D^{2}u(x)\big)\in\Gamma_{k}\quad \text{for all}\ x\in\Omega,
\]
where $\Gamma_{k}$ denotes the G{\aa}rding cone defined by 
\[
\Gamma_{k}:=\left\{\lambda\in\mathbb{R}^{n}:\ \sigma_{j}(\lambda)>0\ \text{for}\   j=1,2,\dots,k\right\}.
\]
They established that for any $\lambda\in\Gamma_{k}$ and each $1\leq i\leq n$, the inequality $\partial\sigma_{k}/\partial\lambda_{i}>0$ holds, and that $\sigma_{k}^{1/k}(\lambda)$ is a concave function in $\Gamma_{k}$.
The $k$-Hessian equation is called degenerate if the nonnegative function $f$ is allowed to vanish at some points in $\overline{\Omega}$, and non-degenerate if $\inf_{\Omega}f>\delta$ for some positive constant $\delta$. 

Chou-Wang \cite{Wang2001} established the Pogorelov type interior estimate for the non-degenerate $k$-Hessian equation, which is a famous result.
More recently, Jiao-Jiao \cite{Jiao2024} extended this result to the degenerate case under the condition
\[
f^{\frac{1}{k-1}}\in C^{1,1}(\overline{\Omega}). 
\]
In this paper, we investigate the Pogorelov type interior estimate for \eqref{eqn1.1} under the condition
\[
f^{\frac{3}{2k-2}}\in C^{2,1}(\overline{\Omega_{0}}),\quad \text{where}\ \Omega\Subset\Omega_{0}.
\]
This new condition was introduced in \cite{Lyu2025}.

\section{Pogorelov type interior estimate}

\subsection{Preliminary}

In this paper, the dependence of constants on dimension $n$ and the number $k$ in \eqref{eqn1.1} is omitted, and a constant depending only on $n$ and $k$ is called universal. 
We adopt the following conventions: $\sigma_{0}(\lambda)=1$; $\sigma_{k}(\lambda)=0$ for $k<0$ and $k>n$; 
\[
\sigma_{k;i_{1}i_{2}\cdots i_{j}}(\lambda)=\sigma_{k}(\lambda)\big|_{\lambda_{i_{1}}=\lambda_{i_{2}}=\cdots=\lambda_{i_{j}}=0};
\]
and $\sigma_{k;i_{1}i_{2}\cdots i_{j}}=0$ if $i_{r}=i_{s}$ for some $1\leq r<s\leq j$. 
Some fundamental properties of $\sigma_k$ are listed below.

\begin{lemma}[\!\cite{Hardy1952,Wang2009}]\label{thm4.1}
For any $k\in\{2,3,\dots,n\}$, the following inequalities and identity hold for any $n\times n$ real symmetric matrix $A$ with  $\lambda(A)\in\Gamma_{k}$:
\begin{align*}
&\sigma_{k}^{\frac{1}{k}}(\lambda(A))\leq C_{1}\sigma_{k-1}^{\frac{1}{k-1}}(\lambda(A));\\
&\sigma_{k-1}(\lambda(A))\geq C_{2}\sigma_{1}^{\frac{1}{k-1}}(\lambda(A))\sigma_{k}^{\frac{k-2}{k-1}}(\lambda(A));
\\
&\sum_{i=1}^{n}\sigma_{k-1;i}(\lambda(A))=(n-k+1)\sigma_{k-1}(\lambda(A)),
\end{align*}
where $C_{1}>0$ and $C_{2}>0$ are universal constants.
\end{lemma}

\begin{lemma}[Section 3 of \cite{Wang2001}]\label{thm2.3}
Assume $\lambda\in\overline{\Gamma_{k}}$ and $\lambda_{1}\geq\lambda_{2}\geq\dots\geq\lambda_{n}$.
Then the following inequalities hold
\begin{equation}\label{eqn7.8}
\begin{cases}
\sigma_{k-1;n}(\lambda)\geq\sigma_{k-1;n-1}(\lambda)\geq\dots\geq\sigma_{k-1;1}(\lambda)\geq0;\\
\sigma_{k-1;k}(\lambda)\geq\theta\sigma_{k-1}(\lambda);\\
\lambda_{1}\sigma_{k-1;1}(\lambda)\geq\theta\sigma_{k}(\lambda),
\end{cases}
\end{equation}
where $\theta>0$ is a universal constant. 
Moreover, for any $\delta\in(0,1)$, there exists $\varepsilon>0$ such that if 
\[
\sigma_{k}(\lambda)\leq\varepsilon\lambda_{1}^{k}\quad \text{or}\quad |\lambda_{i}|\leq\varepsilon\lambda_{1}\quad \text{for}\ i=k+1,k+2,\dots,n,
\]
then the following inequality holds
\begin{equation}\label{eqn7.9}
\lambda_{1}\sigma_{k-1;1}(\lambda)\geq(1-\delta)\sigma_{k}(\lambda). 
\end{equation}
\end{lemma}
\par
\vspace{2mm}

For an $n\times n$ real symmetric matrix $A$, define 
\[
F(A):=\sigma_{k}^{\frac{1}{k}}(\lambda(A))\quad \text{and}\quad F^{ij}(A):=\frac{\partial F}{\partial A_{ij}}.
\]
Then it follows from Lemma \ref{thm4.1} and equation \eqref{eqn1.1} that 
\begin{align}
\operatorname{tr}F^{ij}\big(D^{2}u\big)=\frac{1}{k}\sigma_{k}^{\frac{1}{k}-1}\sum_{i=1}^{n}\sigma_{k-1;i}
&=\frac{n-k+1}{k}\sigma_{k}^{\frac{1}{k}-1}\sigma_{k-1}\label{eqn7.10}\\
&\geq\frac{1}{C}\max\left\{1,\sigma_{1}^{\frac{1}{k-1}}f^{\frac{-1}{k(k-1)}}\right\},\label{eqn4.1}
\end{align}
where $C>0$ is a universal constant.
For convenience, define 
\[
g:=f^{\frac{3}{2k-2}}.
\]
We will need the following properties of positive functions.

\begin{lemma}[Lemma 2.1 of \cite{Lyu2025}]\label{thm2.1}
Let $g$ be a function satisfying $g>0$ in $\Omega_{0}$, where $\Omega\Subset\Omega_{0}$.
We have  

$(\romannumeral1)$ If $g\in C^{1,1}(\overline{\Omega_{0}})$, then 
\begin{equation}\label{eqn7.1}
\frac{\left|\nabla g(x)\right|^{2}}{g(x)}\leq K,\quad \forall x\in\overline{\Omega},
\end{equation}
where $K$ depends on $\|g\|_{C^{1,1}(\overline{\Omega_{0}})}$ and $\operatorname{dist}(\Omega,\partial\Omega_{0})$, but is independent of $\inf_{\Omega_{0}} g$.

$(\romannumeral2)$ If $g\in C^{2,1}(\overline{\Omega_{0}})$, then for any $\alpha<1/2$, 
\begin{equation}\label{eqn7.2}
\partial_{ee}g(x)-\alpha\frac{\left|\partial_{e}g(x)\right|^{2}}{g(x)}\geq-Kg^{\frac{1}{3}}(x),\quad \forall x\in\overline{\Omega},\ e\in\mathbb{S}^{n-1},
\end{equation}
where $K$ depends on $\alpha$, $\|g\|_{C^{2,1}(\overline{\Omega_{0}})}$, and  $\operatorname{dist}(\Omega,\partial\Omega_{0})$, but is independent of $\inf_{\Omega_{0}} g$.
\end{lemma}

\subsection{Main results}

\begin{theorem}\label{thm2.4}
Let $f\geq0$ in $\Omega_{0}$ with $f^{3/(2k-2)}\in C^{2,1}(\overline{\Omega_{0}})$, where $\Omega\Subset\Omega_{0}$. 
Let $k\geq5$. 
Assume that there exists a $k$-admissible function $w\in C^{1,1}(\overline{\Omega})$ such that 
\[
w\geq u\ \ \text{in}\ \Omega\quad \text{and}\quad w=u\ \ \text{on}\ \partial\Omega. 
\]
Then the $k$-admissible solution $u\in C^{3,1}(\overline{\Omega})$ of equation  \eqref{eqn1.1} satisfies the estimate 
\[
(w-u)^{k-1}\left|D^{2}u\right|\leq C,
\]
where the constant $C>0$ depends on $\|u\|_{C^{1}(\overline{\Omega})}$, $\|w\|_{C^{1,1}(\overline{\Omega})}$, $\|f^{3/(2k-2)}\|_{C^{2,1}(\overline{\Omega_{0}})}$, and $\operatorname{dist}(\Omega,\partial\Omega_{0})$, but is independent of $\inf_{\Omega_{0}}f$.
\end{theorem}
\begin{proof}
For convenience, we let $K$ denote a positive constant which may vary from line to line.
We follow the proof of the well-known Theorem 1.5 in \cite{Wang2001}.
Consider the auxiliary function 
\[
G(x,\xi):=\rho^{\beta}\varphi\left(\frac{1}{2}|Du|^{2}+\frac{1}{2}|x|^{2}\right)u_{\xi\xi}\quad \text{in}\ \Omega\times\mathbb{S}^{n-1},  
\]
where $\beta>0$ is to be determined, $\varphi(t)=(1-\frac{t}{M})^{-1/8}$, $M=2\sup_{x\in\Omega}(|Du|^{2}+|x|^{2})$, and 
\[
\rho=w-u. 
\]
Suppose that $G$ attains its maximum at an interior point $x_{0}$ in the direction $\xi=e_{1}$. 
Thus 
\[
H:=\beta\log\rho+\log\varphi\left(\frac{1}{2}|Du|^{2}+\frac{1}{2}|x|^{2}\right)+\log u_{11}
\]
also attains its maximum at $x_{0}$.
Assume that $D^{2}u$ is already diagonal at $x_{0}$ with $u_{11}\geq u_{22}\geq\dots\geq u_{nn}$. 
Then at $x_{0}$, 
\begin{equation}\label{eqn2.1}
0=H_{i}=\beta\frac{\rho_{i}}{\rho}+\frac{\varphi_{i}}{\varphi}+\frac{u_{11i}}{u_{11}}\quad \text{for}\ i=1,2,\dots,n,
\end{equation}
and 
\begin{align}
0&\geq\sum_{i=1}^{n}F^{ii}H_{ii}\nonumber\\
&=\sum_{i=1}^{n}\beta F^{ii}\left[\frac{\rho_{ii}}{\rho}-\frac{\rho_{i}^{2}}{\rho^{2}}\right]
+\sum_{i=1}^{n}F^{ii}\left[\frac{\varphi_{ii}}{\varphi}-\frac{\varphi_{i}^{2}}{\varphi^{2}}\right]
+\sum_{i=1}^{n}F^{ii}\left[\frac{u_{11ii}}{u_{11}}-\frac{u_{11i}^{2}}{u_{11}^{2}}\right].\label{eqn2.2} 
\end{align}
By our special choice of $\rho$, we have at $x_{0}$, 
\begin{equation}\label{eqn2.7}
\sum_{i}F^{ii}\rho_{ii}=\sum_{i,j}F^{ij}\big(D^{2}u\big)(w_{ij}-u_{ij})\geq-\sum_{i,j}F^{ij}\big(D^{2}u\big)u_{ij}=-f^{\frac{1}{k}}.
\end{equation}
At $x_{0}$, we also have  
\begin{align}
\sum_{i}F^{ii}\left[\frac{\varphi_{ii}}{\varphi}-3\frac{\varphi_{i}^{2}}{\varphi^{2}}\right]&=\left(\frac{\varphi''}{\varphi}-3\frac{\varphi'^{2}}{\varphi^{2}}\right)\sum_{i}F^{ii}(u_{i}u_{ii}+x_{i})^{2}\nonumber\\
&\hspace{4.5mm}+\frac{\varphi'}{\varphi}\left(\sum_{i,j}u_{j}F^{ii}u_{iij}+\sum_{i}F^{ii}u_{ii}^{2}+\operatorname{tr}F^{ij}\right)\nonumber\\
&\geq\gamma\left(\sum_{i}F^{ii}u_{ii}^{2}+\operatorname{tr}F^{ij}\right)-K\left|D\Big(f^{\frac{1}{k}}\Big)\right|,\label{eqn2.5}
\end{align}
where the specific choice $\varphi(t)=(1-\frac{t}{M})^{-1/8}$ ensures that 
\[
K\geq\frac{\varphi''}{\varphi}-3\frac{\varphi'^{2}}{\varphi^{2}}\geq0\quad \text{and}\quad K\geq\frac{\varphi'}{\varphi}\geq\frac{1}{8M}=:\gamma>0,
\]
since $t\in[0,M/4]$ by the definition of $M$.
It follows from \eqref{eqn2.1} that 
\begin{equation}\label{eqn2.3}
\frac{u_{11i}}{u_{11}}=-\left(\frac{\varphi_{i}}{\varphi}+\beta\frac{\rho_{i}}{\rho}\right)\quad \text{for}\ i=1,2,\dots,n, 
\end{equation}  																
and 
\begin{equation}\label{eqn2.4}
\frac{\rho_{i}}{\rho}=-\frac{1}{\beta}\left(\frac{\varphi_{i}}{\varphi}+\frac{u_{11i}}{u_{11}}\right)\quad \text{for}\ i=1,2,\dots,n.
\end{equation}

We consider two cases separately. 

\par
\vspace{2mm}
\textbf{Case 1}: \emph{$u_{kk}<\varepsilon u_{11}$, where $\varepsilon$ is a small positive constant to be determined.}
				
Putting \eqref{eqn2.3} for $i=1$ and \eqref{eqn2.4} for $i=2,3,\dots,n$ into \eqref{eqn2.2}, we obtain 
\begin{align*}
0&\geq\left\{\sum_{i=1}^{n}\left[\beta F^{ii}\frac{\rho_{ii}}{\rho}+F^{ii}\left(\frac{\varphi_{ii}}{\varphi}-3\frac{\varphi_{i}^{2}}{\varphi^{2}}\right)\right]-\beta(1+2\beta)F^{11}\frac{\rho_{1}^{2}}{\rho^{2}}\right\}\\
&\hspace{4.5mm}+\left\{\sum_{i=1}^{n}\frac{F^{ii}u_{11ii}}{u_{11}}-\left(1+\frac{2}{\beta}\right)\sum_{i=2}^{n}F^{ii}\frac{u_{11i}^{2}}{u_{11}^{2}}\right\}\\
&=:I_{1}+I_{2}.
\end{align*} 
By \eqref{eqn2.7} and \eqref{eqn2.5}, we have 
\begin{align}
I_{1}&\geq\gamma\big(F^{11}u_{11}^{2}+\operatorname{tr}F^{ij}\big)-\frac{KF^{11}}{\rho^{2}}-\frac{\beta f^{\frac{1}{k}}}{\rho}-K\left|D\Big(f^{\frac{1}{k}}\Big)\right|\nonumber\\
&=F^{11}u_{11}^{2}\left(\gamma-\frac{K}{(\rho u_{11})^{2}}\right)+\gamma\operatorname{tr}F^{ij}-\frac{\beta f^{\frac{1}{k}}}{\rho}-K\left|D\Big(f^{\frac{1}{k}}\Big)\right|\nonumber\\
&\geq\gamma\operatorname{tr}F^{ij}-\frac{\beta f^{\frac{1}{k}}}{\rho}-K\left|D\Big(f^{\frac{1}{k}}\Big)\right|,\label{eqn2.21} 
\end{align}
provided that $\rho u_{11}$ is sufficiently large; otherwise, the desired estimate holds trivially.
From the equality (4.3) in \cite{Wang2001} and the concavity of $F$, we derive  
\begin{align}
u_{11}I_{2}&\geq \partial_{11}\Big(f^{\frac{1}{k}}\Big)+\sum_{i,j=1}^{n}\frac{1}{k}\sigma_{k}^{\frac{1}{k}-1}\sigma_{k-2;ij}u_{1ij}^{2}-\left(1+\frac{2}{\beta}\right)\sum_{i=2}^{n}\frac{1}{k}\sigma_{k}^{\frac{1}{k}-1}\sigma_{k-1;i}\frac{u_{11i}^{2}}{u_{11}}\nonumber\\
&\geq \partial_{11}\Big(f^{\frac{1}{k}}\Big)+\frac{2}{k}f^{\frac{1}{k}-1}\sum_{i=2}^{n}\left(\sigma_{k-2;1i}-\left(\frac{1}{2}+\frac{1}{\beta}\right)\frac{\sigma_{k-1;i}}{u_{11}}\right)u_{11i}^{2}.\label{eqn3.1}
\end{align}
Taking $\beta\geq4$ and using \eqref{eqn7.9}, we can fix $\varepsilon$ sufficiently small such that 
\begin{equation}\label{eqn3.2}
\sigma_{k-2;1i}-\left(\frac{1}{2}+\frac{1}{\beta}\right)\frac{\sigma_{k-1;i}}{u_{11}}\geq0,\quad \forall2\leq i\leq n.
\end{equation}
It follows from \eqref{eqn2.21}--\eqref{eqn3.2} that 
\begin{equation}\label{eqn3.3}
0\geq I_{1}+I_{2}\geq\gamma\operatorname{tr}F^{ij}-\frac{\beta f^{\frac{1}{k}}}{\rho}-K\left|D\Big(f^{\frac{1}{k}}\Big)\right|+\frac{1}{u_{11}}\partial_{11}\Big(f^{\frac{1}{k}}\Big).
\end{equation}

By \eqref{eqn7.1}, we get 
\begin{equation}\label{eqn3.4}
\left|D\Big(f^{\frac{1}{k}}\Big)\right|=\left|D\Big(g^{\frac{2k-2}{3k}}\Big)\right|\leq Kg^{\frac{2k-2}{3k}-\frac{1}{2}}=Kf^{\frac{1}{4(k-1)}-\frac{1}{k(k-1)}}.
\end{equation}
By \eqref{eqn7.2}, we get 
\begin{equation}\label{eqn3.5}
\partial_{11}\Big(f^{\frac{1}{k}}\Big)=\frac{2k-2}{3k}g^{\frac{-k-2}{3k}}\left(\partial_{11}g-\frac{k+2}{3k}\frac{\left|\partial_{1}g\right|^{2}}{g}\right)\geq-Kg^{-\frac{2}{3k}}=-Kf^{\frac{-1}{k(k-1)}},
\end{equation}
since 
\[
\frac{k+2}{3k}<\frac{1}{2},\quad \forall k\geq5.
\] 
It follows from \eqref{eqn3.3}--\eqref{eqn3.5} and \eqref{eqn4.1} that 
\begin{align*}
0&\geq\frac{\gamma }{C}\sigma_{1}^{\frac{1}{k-1}}f^{\frac{-1}{k(k-1)}}-\frac{K}{\rho}-\left(K+\frac{K}{u_{11}}\right)f^{\frac{-1}{k(k-1)}}\\
&=\frac{1}{\rho}\left[\left(\frac{\gamma }{C}\left(\rho^{k-1}\sigma_{1}\right)^{\frac{1}{k-1}}-\rho K-\frac{\rho K}{u_{11}}\right)f^{\frac{-1}{k(k-1)}}-K\right],
\end{align*}
where $C>0$ is a universal constant.
Therefore, since $\lambda(D^{2}u)\in\Gamma_{2}$, we obtain 
\[
\rho^{k-1}u_{11}\leq\rho^{k-1}\sigma_{1}\leq K.
\]
\par
\vspace{2mm}
\textbf{Case 2}: \emph{$u_{kk}\geq\varepsilon u_{11}$.}

Putting \eqref{eqn2.3} into \eqref{eqn2.2} yields 
\begin{equation}\label{eqn2.8}
0\geq\beta F^{ii}\left[\frac{\rho_{ii}}{\rho}-(1+2\beta)\frac{\rho_{i}^{2}}{\rho^{2}}\right]
+F^{ii}\left[\frac{\varphi_{ii}}{\varphi}-3\frac{\varphi_{i}^{2}}{\varphi^{2}}\right]
+F^{ii}\frac{u_{11ii}}{u_{11}}. 
\end{equation}
By the concavity of $F$, 
\begin{equation}\label{eqn2.9}
F^{ii}u_{11ii}\geq\partial_{11}\Big(f^{\frac{1}{k}}\Big). 
\end{equation}
It follows from \eqref{eqn2.5}, \eqref{eqn7.8}, and \eqref{eqn7.10} that 
\begin{align}\label{eqn2.10}
F^{ii}\left[\frac{\varphi_{ii}}{\varphi}-3\frac{\varphi_{i}^{2}}{\varphi^{2}}\right]&\geq\gamma F^{kk}u_{kk}^{2}-K\left|D\Big(f^{\frac{1}{k}}\Big)\right|\nonumber\\
&\geq\gamma\frac{1}{k}\sigma_{k}^{\frac{1}{k}-1}\theta_{1}\sigma_{k-1}\varepsilon^{2}u_{11}^{2}-K\left|D\Big(f^{\frac{1}{k}}\Big)\right|\nonumber\\
&=\gamma\varepsilon^{2}\theta_{2}u_{11}^{2}\operatorname{tr}F^{ij}-K\left|D\Big(f^{\frac{1}{k}}\Big)\right|,
\end{align}
where $\theta_{1}>0$ and $\theta_{2}>0$ are universal constants. 
Combining \eqref{eqn2.7} and \eqref{eqn2.8}--\eqref{eqn2.10}, we derive
\begin{equation}\label{eqn3.6}
0\geq-\frac{\beta f^{\frac{1}{k}}}{\rho}-K\frac{\operatorname{tr}F^{ij}}{\rho^{2}}+\gamma\varepsilon^{2}\theta_{2}u_{11}^{2}\operatorname{tr}F^{ij}-K\left|D\Big(f^{\frac{1}{k}}\Big)\right|+\frac{1}{u_{11}}\partial_{11}\Big(f^{\frac{1}{k}}\Big).
\end{equation}
Using \eqref{eqn3.6}, \eqref{eqn3.4}, \eqref{eqn3.5}, and \eqref{eqn4.1}, we obtain 
\begin{align}
0&\geq\left(\gamma\varepsilon^{2}\theta_{2}(\rho u_{11})^{2}-K\right)\operatorname{tr}F^{ij}-\rho\beta f^{\frac{1}{k}}-\rho^{2}Kf^{\frac{-1}{k(k-1)}}-\frac{\rho^{2}K}{u_{11}}f^{\frac{-1}{k(k-1)}}\nonumber\\
&\geq\left(\frac{\gamma\varepsilon^{2}\theta_{2}(\rho u_{11})^{2}-K}{C}\sigma_{1}^{\frac{1}{k-1}}-K\right)f^{\frac{-1}{k(k-1)}}-K,\label{eqn6.1}
\end{align}
provided that $\rho u_{11}$ is sufficiently large; otherwise, the desired estimate holds trivially.
Here, $C>0$ is a universal constant.
From \eqref{eqn6.1}, we conclude 
\[
\rho u_{11}\leq K.
\]

Therefore, taking $\beta=k-1$, we obtain $G(x_{0})\leq K$. 
The proof of Theorem \ref{thm2.4} is complete. 
\end{proof}

\end{document}